\documentclass[11pt,a4paper,final]{article}

\usepackage{a4}
\usepackage{amsfonts}                  
\usepackage{amsmath}                   
\usepackage{amsthm}
\usepackage{color}
\usepackage{amssymb}
\usepackage{enumerate}
\usepackage{soul}

\usepackage{amssymb}
\usepackage{enumerate}
\usepackage{soul}
\usepackage{pgfplots}
\pgfplotsset{compat=1.15}
\usepackage{mathrsfs}
\usetikzlibrary{arrows}

\DeclareMathOperator{\eco}{econv}	

\DeclareMathOperator{\conv}{conv}

\DeclareMathOperator{\cl}{cl}

\DeclareMathOperator{\epco}{\normalfont{e}^\prime conv} 	
 	
\DeclareMathOperator{\ep}{\text{e}^\prime} 	
\DeclareMathOperator{\dom}{dom}					
\DeclareMathOperator{\epi}{epi}

\newcommand{\Ramp}{\overline{\mathbb{R}}}
\newcommand{\R}{\mathbb{R}}

\newcommand{\RTD}{\mathbb{R}^{(T)}}
\newcommand{\ci}{\left\langle}		
\newcommand{\cd}{\right\rangle}

\newcommand{\Db}{\overline{D}}

\theoremstyle{plain}
\newtheorem{theorem}{Theorem}[section]
\newtheorem{lemma}[theorem]{Lemma}
\newtheorem{corollary}[theorem]{Corollary}
\newtheorem{proposition}[theorem]{Proposition}

\theoremstyle{definition}
\newtheorem{definition}[theorem]{Definition}
\newtheorem{example}[theorem]{Example}

\theoremstyle{remark}
\newtheorem{remark}{Remark}

\begin{document}

\title{Lagrange duality on DC evenly convex optimization problems via a generalized conjugation scheme}

\author{Maria Dolores Fajardo\thanks{Department of Mathematics, University of Alicante, Alicante, Spain, md.fajardo@ua.es} \and
Jos\'e Vidal\thanks{Department of Physics and Mathematics, University of Alcalá, Alcalá de Henares, Spain, j.vidal@uah.es}}

\maketitle

\begin{abstract}
In this paper we study how Lagrange duality is connected to optimization problems whose objective function is the difference of two convex functions, briefly called DC problems. We present two Lagrange dual problems, each of them obtained via a different approach. While one of the duals corresponds to the standard formulation of the Lagrange dual problem, the other is written in terms of conjugate functions. When one of the involved functions in the objective is evenly convex, both problems are equivalent, but this relation is no longer true in the general setting. For this reason, we study conditions ensuring not only weak, but also zero duality and strong duality between the primal and one of the dual problems written using conjugate functions. For the other dual, and due to the fact that weak duality holds by construction, we just develop conditions for zero duality gap and strong duality between the primal DC problem and its (standard) Lagrange dual problem.
Finally, we characterize weak and strong duality together with zero duality gap between the primal problem and its Fenchel-Lagrange dual following techniques used throughout the manuscript.
\end{abstract}



\section{Introduction}
\label{sec:Intro}
Conjugate duality is a branch within the area of optimization that has attained a lot of attention in mathematics. One of the reasons of its success is that, given an optimization problem of interest, called \textit{primal} problem, it makes it possible to analyze an alternative problem, called \textit{dual}, to solve the pursued problem. To get a deep introduction about basic tools and a complete theoretical perspective, we refer the reader, for instance, to the textbooks \cite{BC2011,B2010,R1970,Z2002} among many others. As recent applications of duality theory, we mention that this approach has been used in \cite{HHKSV19} in the context of inverse problems, \cite{BHSTV21} in optimization problems posed on Riemannian manifolds, \cite{ML2005} in the context of economics or \cite{CP2011} in the development of the well-known primal-dual algorithm call Chambolle-Pock.

In this paper, we focus our attention both in a particular type of duality, which is Lagrange duality, and in a distinctive type of optimization problems, called DC problems. For a throughout introduction on them, we refer the reader to \cite{HT1999, TA1997,TD2018,Oli2020,TD2023}. These problems allow to mathematically model optimization problems belonging to areas like, for instance, data visualization \cite{CGR2018}, and they are also of interest from the algorithmic point of view, as one can appreciate from \cite{ACV2022} or \cite{BFSS2024}. Regarding Lagrange duality, it has already been explored not only in the classical setting, using Fenchel conjugate teory, by \cite{JDL2004,BW2006A,BGW2008TL}, but also using a more general conjugation scheme, the $c$-conjugation, based on generalized convex conjugation in \cite{FVR2016} and \cite{FV2016SSD}, where the authors develop sufficient conditions ensuring strong duality between a convex primal problem and its Lagrange dual. To get an entire perspective using this generalized conjugation scheme into duality theory, see \cite{FGRVP2020}. To make profit of this conjugation scheme, the involved functions must be evenly convex. Following \cite{RVP2011}, a function is evenly convex if its epigraph is an evenly convex set, a set which can be expressed as the intersection of an arbitrary family of open half-spaces. These sets were introduced by Fenchel in \cite{F1952} to extend polarity theory to non-closed and convex sets.

In the novel work \cite{FV2023}, two Fenchel dual problems for a DC optimization primal one were built  by means of $c$-conjugation, being both equivalent under even convexity assumption of one of the functions in the primal problem. Characterizations for weak, strong and stable strong duality were presented. That work has motivated us to continue with the study of Lagrange duality, making use of the Lagrange dual problem developed in \cite{FVR2016} for a convex optimization primal one and the way that, again under even convexity assumptions, this dual problem can be reformulated in an equivalent form expressed in terms of the $c$-conjugates of the involved functions in the primal problem. This tecnique, applied to Fenchel conjugation, is used in \cite{FLLY2013}. As it happened with Fenchel duality, weak duality is not guaranteed with this reformulated dual problem if the even convexity assumption does not hold. This issue serves as a starting point to develop conditions ensuring not only weak and strong duality, but also zero duality gap as well for this primal-dual pair. 

On the other hand and, since in \cite{FVR2016} there is no analysis about zero duality gap between the primal problem and its standard Lagrange dual, this paper recaps this situation providing necessary and sufficient conditions for both strong duality and zero duality gap. As it will be shown, these conditions are written using the epigraphs of the $c$-conjugate functions of the involved functions in the primal problem, so they belong to the class of closedness-type regularity conditions. Additionally, we also compare strong dualities of both primal-dual pairs using the theoretical results presented throughout the paper. Last but not least, we study conditions for weak, strong and zero duality gap between the DC primal problem and its Fenchel-Lagrange dual. As it happens in \cite{FV2023} for Fenchel and in the sequel for Lagrange duals, it is also possible to derive an equivalent formulation of the Fenchel-Lagrange dual problem presented in \cite{FV2017}. Under the same even convexity assumption, both dual problems turn out to be equivalent and we characterize weak and strong dualities together with zero duality gap between the new Fenchel-Lagrange dual and the primal problem.

The structure of the paper is as follows. In Section \ref{sec:Pre} we state the necessary results to keep the paper self-contained. Section \ref{sec:Dual_problems} presents two Lagrange dual problems that will be the focus of the work. As it will be shown, both are equivalent under the even convexity of just one of the functions in the objective. Section \ref{sec:Cond_WD} contains conditions for weak and strong duality and zero duality gap between the primal and the Lagrange dual problem which is expressed in terms of conjugate functions. Section \ref{sec:Comp_DL1} characterizes both strong duality and zero duality gap between the primal and its standard Lagrange dual problem, and it also compares strong dualities of both primal-dual pairs derived in Section \ref{sec:Dual_problems}. Section \ref{sec:FL_duality} is devoted to the analysis of an alternative Fenchel-Lagrange dual problem which turns out to be equivalent to the (standard) Fenchel-Lagrange dual under appropiate assumptions. Finally, Section \ref{sec:Conclusions} emphasizes the achieved goals in the paper and concludes the manuscript.

\section{Preliminaries}
\label{sec:Pre}
This section contains the necessary results and concepts that will be used in the sequel. Let $X$ be a non-trivial separated locally convex space, lcs in short, whose topology is the one induced by its continuous dual space $X^*$, i.e., $\sigma(X,X^*)$. We represent by $\ci x,x^*\cd$ the action of $x^*$ at $x$. For a set $D$ in $X$, we denote by $\conv(D)$ and $\cl(D)$ its convex hull and closure, respectively. Throughout we will use the notation $\Ramp:=\R\cup\left\{\pm\infty\right\}$, $\R_+=[0,+\infty[$ and $\R_{++}=]0,+\infty[$.

As indicated in Section \ref{sec:Intro}, evenly convex sets, briefly e-convex, were introduced in \cite{F1952}, studied in terms of their sections and projections in \cite{KMZ2007} and characterized in \cite{DML2002} as follows. Actually this characterization is the way that e-convex sets are defined in \cite{DML2002}, more tratable than the original one.

\begin{definition}[Def.~1, \cite{DML2002}]
\label{def:Econvex}
A set $C\subseteq X$ is e-convex if for every point $x_0\notin C$, there exists $x^*\in X^*$ such that $\ci x-x_0,x^*\cd<0$, for all $x\in C$.
\end{definition}

If $D$ is a set in $X$, its \textit{convex hull}, denoted by $\eco D$, 
is the smallest e-convex set that contains $D$. If $D$ is, moreover, convex, then
\begin{equation*}
	D\subseteq \eco D\subseteq \cl D.
\end{equation*}
The e-convex hull operator is closed under arbitrary intersections and due to Hahn-Banach theorem, every closed or open convex set is e-convex too, being the converse statement false in general; see \cite[Ex.~2.2]{FV2023}.

Given a function $f:X\to\Ramp$, we keep the standard notation for its domain and epigraph, that is,
\begin{align*}
	\dom f&:=\left\{ x\in X~:~f(x)<+\infty\right\},\\
	\epi f&:=\left\{(x,\alpha)\in X\times\R~:~f(x)\leq\alpha\right\},
\end{align*}
and it is \textit{proper} if $f(x)> -\infty$, for all $x \in X$, and its domain is non-empty. Its \textit{lower semicontinuous convex hull}, denoted by $\cl f$, is the function verifying $\epi\cl f=\cl(\epi f)$, and $f$ is called \textit{lower semicontinuous}, lsc, in breaf, if $f=\cl f$.
In a bit different way, the \textit{e-convex hull} of a proper function $f$, briefly denoted by $\eco f$, is its largest e-convex minorant. Let us observed that it cannot be defined as that function verifying $\epi\eco f=\eco(\epi f)$.

\begin{example}
Let us consider the function $f:\R \rightarrow \Ramp$
\begin{equation*}
	f(x)=\left\{
	\begin{aligned}
		x,~&~\text{if } x > 0,\\
		1,~&~\text{if } x = 0,\\
		+\infty, ~&~ \text{otherwise.}
	\end{aligned}
	\right.
\end{equation*}
Then $\eco(\epi f)=\left\lbrace(x,x): x\geq 0 \right\rbrace\setminus \left\lbrace(0,0)\right\rbrace$, and there is no function whose epigraph is this e-convex set; see Figure \ref{fig:Example_econv}. 
\begin{figure}
\definecolor{zzttqq}{rgb}{0.6,0.2,0.}
\definecolor{ttqqqq}{rgb}{0.2,0.,0.}
\begin{tikzpicture}[line cap=round,line join=round,>=triangle 45,x=1.0cm,y=1.0cm]
\clip(-6.5,0) rectangle (4,6);
\fill[line width=2.pt,color=zzttqq,fill=zzttqq,fill opacity=0.10000000149011612] (-1.,1.) -- (-1.,4.683792839178872) -- (2.6413367437101147,4.683792839178872) -- (2.6413367437101147,1.) -- cycle;
\draw [->,line width=1.pt] (-1.,1.) -- (-1.,5.);
\draw [->,line width=1.pt] (-1.,1.) -- (3.,1.);
\draw (-0.3090098537597738,3.4613384720997162) node[anchor=north west] {$econv(epi f)$};
\begin{scriptsize}
\draw [color=ttqqqq] (-1.,1.) circle (3.0pt);
\end{scriptsize}
\end{tikzpicture}
\caption{Set from Example 2.2}
\label{fig:Example_econv}
\end{figure}
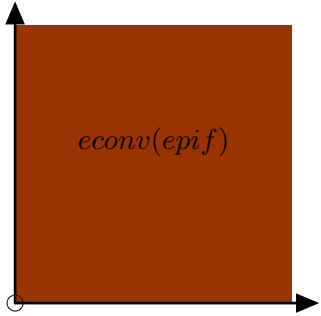
\end{example}
 The class of lsc functions is a strictly contained subclass in the class of e-convex functions, see \cite[Ex.~2.1]{FV2017}, which is an expected resulted, knowing the relationship between closed convex sets and e-convex sets.\\
 
In the rest of this preliminary section, we state the conjugation scheme that we will use throughout the paper, which is called the $c$-conjugation scheme; see \cite{MLVP2011}. It is based on the generalized convex conjugation theory presented by Moreau in \cite{Mor1970} and it uses a particular pair of coupling functions. Let $f:X\to\Ramp$ be a proper function and set $W:=X^*\times X^*\times \R$. The coupling function used in \emph{$c$-conjugacy} is $c:X\times W\to\Ramp$
\begin{equation*}
	c(x,(x^*,y^*,\alpha))=\left\{
	\begin{aligned}
		\ci x,x^*\cd, &~~~ \text{if } \ci x,y^*\cd<\alpha\\
		+\infty, & ~~~ \text{otherwise.}
	\end{aligned}
	\right.
\end{equation*}
The \emph{$c$-conjugate} of a function $f:X\to\Ramp$ is defined as
\begin{equation*}
	f^c(x^*,y^*,\alpha):=\sup_{X}\left\{ c(x,(x^*,y^*,\alpha))-f(x)\right\}.
\end{equation*}
Observe that $c$-conjugation provides different results than standard Fenchel conjugation, see \cite[Ex.~2.5]{FV2016SSD}. In fact, 
\begin{equation*}
	f^c(x^*,y^*,\alpha)= \left \{
	\begin{aligned}
		f^*(x^*), &~~~ \text{if } \dom f \subseteq H_{y^*,\alpha}^{-},\\
		+\infty, & ~~~ \text{otherwise,}
	\end{aligned}
	\right.
\end{equation*}
where $f^*:X^* \to \Ramp$ is the Fenchel conjugate of $f$ and $H_{y^*,\alpha}^{-}$ is the open half-space in $X$ defined as
\begin{equation*}
	H_{y^*,\alpha}^{-}:=\{x\in X : \ci x,y^*\cd < \alpha \}.
\end{equation*}
This conjugation scheme is complemented by the use of the coupling function $c^\prime:W\times X\to\Ramp$ defined as
\begin{equation*}
	c^\prime((x^*,y^*,\alpha),x):=c(x,(x^*,y^*,\alpha)),
\end{equation*}
which allows the following definition of the \emph{$c^\prime$-conjugate} of a proper function $h:W\to\Ramp$ as
\begin{equation*}
	h^{c^\prime}(x):=\sup_{W}\left\{c^{\prime}((x^*,y^*,\alpha),x)-h(x^*,y^*,\alpha)\right\}.
\end{equation*}
This conjugation pattern gives priority to $-\infty$, i.e., the sign convention is 
\begin{equation*}
	(+\infty)+(-\infty)=(-\infty)+(+\infty)=(+\infty)-(+\infty)=(-\infty)-(-\infty)=-\infty.
\end{equation*}
Based on the generalized convex duality theory, see \cite{Mor1970}, functions $c(\cdot,(x^*,y^*,\alpha))-\beta :X \rightarrow \Ramp$, with $(x^*,y^*,\alpha) \in W $ and $\beta \in \R$, are called \emph{c-elementary}, and functions $c^{\prime}(\cdot,x)-\beta : W \rightarrow \Ramp$, with $x\in X$ and $\beta \in \mathbb{R}$, are called \emph{c}$^{\prime }$-\emph{elementary}. In \cite{MLVP2011}, it is shown that any proper e-convex function $f: X \rightarrow \Ramp$ is the pointwise supremum of a set of $c$-elementary functions. Extending this concept, \cite{FVR2012} introduced the \emph{$\ep$-convex functions} as convex functions $g:W \rightarrow \Ramp$ which can be expressed as the pointwise supremum of a set of $c^\prime$-elementary functions. The \emph{$\ep$-convex hull} of any function $g:W\rightarrow\Ramp$, $\epco g$, is the largest $\ep$-convex minorant function  of $g$. The following theorem from \cite{ML2005} is the counterpart of Fenchel-Moreau theorem for e-convex and $\ep$-convex functions.
 
\begin{theorem}[Prop. 6.1, 6.2, Cor. 6.1,~\cite{ML2005}]
\label{thm:Theorem_charac}
Let $f:X\rightarrow \overline{\mathbb{R}}$ and $g:W\rightarrow \overline{\mathbb{R}}$. Then
\begin{itemize}
\item[(i)] $f^{c}$ is e$^{\prime }$-convex$;$ $g^{c^{\prime }}$ is e-convex.

\item[(ii)] If $f$ has a proper e-convex minorant, then $\eco f=f^{cc^{\prime }}$; $ \epco g=g^{c^{\prime }c}$.

\item[(iii)] If $f$ does not take on the value $-\infty $, then $f$ is e-convex if and only if $f=f^{cc^{\prime }};g$ is e$^{\prime }$-convex if and only if $g=g^{c^{\prime }c}$.

\item[(iv)] $f^{cc^{\prime }}\leq f;$ $g^{c^{\prime }c}\leq g$. 
\end{itemize}
\end{theorem}
According to $(iii)$, it makes sense to say that, if $f$ does not take on the value $-\infty $, then $f$ is e-convex at a point $x_0 \in X$ if and only if $f(x_0)=f^{cc^{\prime }}(x_0)$ and $g$ is e$^{\prime }$-convex at a point $(x_0^*,y_0^*,\alpha_0) \in W$ if and only if $g(x_0^*,y_0^*,\alpha_0)=g^{c^{\prime }c}(x_0^*,y_0^*,\alpha_0)$.\\

We finish with the following definition of a kind of sets that complements the family of e-convex sets. Their main properties were developed in \cite{FV2020} and they were used in  \cite{F2015,FVR2016,FV2018,FV2016SSD} in the study of regularity conditions for different pairs of primal-dual problems.

\begin{definition}[Def.2,~\cite{FVR2012}]
A set $D \subset W \times \R$ is \textit{$\ep$-convex} if there exists an $\ep$-convex function $k : W \to \R$ such that $D = \epi k$. The \textit{$\ep$-convex hull} of an arbitrary set $D \subset W \times \R$ is defined as the smallest
$\ep$-convex set containing $D$, and it will be denoted by $\epco D$.
\end{definition}

\section{Dual problems}
\label{sec:Dual_problems}
Consider the following optimization problem 
\begin{equation}
	\label{eq:Primal_problem}
	\tag{$P$}
	\inf_A \left \lbrace f(x)-g(x) \right\rbrace
	\end{equation}
where 
\begin{equation}
	\label{eq:Set_A}
	A=\left\{x\in X\,:\, h_t(x)\leq 0, ~t\in T\right\}
\end{equation}
and $f,g, h_t:X\to\Ramp$ are proper convex functions, for all $t\in T$, being $T$ an arbitrary index set. 
Using the perturbational approach, in \cite{FV2023} a Fenchel dual problem for $(P)$ is presented. Additionally, and using a convexification technique, another dual problem is obtained, being both problems equivalent whenever the function $g$ in \eqref{eq:Primal_problem} is e-convex. Also characterizations of weak duality, zero duality gap and strong duality are given, as well as a relationship between strong duality for both pairs of primal-dual problems.

Inspired by that work, firstly, in this section, we will introduce two Lagrange dual problems  that are also equivalent in case the function $g$ in \eqref{eq:Primal_problem} is e-convex.
We assume that $f(x)-g(x)=+\infty$ in case $x \not \in \dom f$, which implies that $f-g$ will be a proper function if $\dom f \subseteq \dom g$, and, consequently, for all $(u^*, v^*, \gamma ) \in \dom g^c$, $\dom f \subseteq H_{v^*, \gamma}^-$.

According to \cite{FVR2016}, a dual problem for \eqref{eq:Primal_problem} can be built using the perturbational approach (see, for instance, \cite{ET1976}), by means of the $c$-conjugate of the perturbation function 
\begin{equation}
	\label{eq:Dual_problem_1}
	\tag{$D_L$}
	\sup_{\RTD_+}\inf_{X} \left\{f(x)-g(x)+\lambda h(x)\right\},
\end{equation}
verifying weak duality, i.e., $v(P) \geq v \eqref{eq:Dual_problem_1}$. Here $\RTD$ denotes the space of generalized finite sequences, i.e., a sequence $\lambda=(\lambda_t)_{t\in T}$ belongs to $\RTD$ if only finitely many $\lambda_t$ are different from zero, and ${\RTD_+}$ represents its nonnegative polar cone.
It is worth emphasizing that such a dual problem was obtained in \cite{FVR2016} without restrictions on the involved functions, just proper extended real valued functions. The same weak duality result was shown in \cite[Lem.\,2.2]{MLV1999} using the Fenchel conjugate pattern.

In case the function $g$ is e-convex, being moreover proper, we can use the equality $g=g^{cc^\prime}$, by Theorem \ref{thm:Theorem_charac} $(iii)$, to deduce 
\begin{align*}
	\inf_{X} & \left\{ f(x)-g(x)+\lambda h(x)\right\}\\
	=& \inf_{X} \left\{ f(x)-\sup_{W}\left\{ c(x,(u^*,v^*,\gamma))-g^c(u^*,v^*,\gamma)\right\} +\lambda h(x)\right\}\\
	=& \inf_{W} \left\{ g^c(u^*,v^*,\gamma)-\sup_{X}\left\{ c(x,(u^*,v^*,\gamma))-(f(x)+\lambda h(x))\right\} \right\}\\
	=& \inf_{W}\left\{g^c(u^*,v^*,\gamma)-(f+\lambda h)^c(u^*,v^*,\gamma)\right\}.
\end{align*} 
It leads to reformulate the dual problem \eqref{eq:Dual_problem_1} as
\begin{equation}
	\label{eq:Dual_problem_2}
	\tag{$\overline{D}_L$}	
	\sup_{\RTD_+}\inf_{ W} \left\{g^c(u^*,v^*,\gamma)-(f+\lambda h)^c(u^*,v^*,\gamma)\right\}.
	\end{equation}
Nevertheless, without assuming the e-convexity of $g$, the dual problems \eqref{eq:Dual_problem_1} and \eqref{eq:Dual_problem_2} may not be equivalent and, even more, weak duality might not hold for the dual pair \eqref{eq:Primal_problem}$-$\eqref{eq:Dual_problem_2} as we can see in the following example.

\begin{example}
\label{ex:No_equivalence_L}
Let $T$ be a singleton and $f,g,h:X\to\Ramp$ defined by
\begin{equation*}
	f(x)=\left\{
	\begin{aligned}
		x,~&~\text{if } x\geq 0,\\
		+\infty, ~&~ \text{otherwise;}
	\end{aligned}
	\right.
	\hspace{0.35cm}
	g(x)=\left\{
	\begin{aligned}
		x,~&~\text{if } x> 0,\\
		1,~&~\text{if } x=0,\\
		+\infty, ~&~ \text{otherwise;}
	\end{aligned}
	\right.
	\hspace{0.35cm}
	h(x)=\left\{
	\begin{aligned}
		-x,~&~\text{if } x\geq 0,\\
		+\infty, ~&~ \text{otherwise.}
	\end{aligned}
	\right.
\end{equation*}

Since $A=\left\lbrace x \in \R: h(x) \leq 0 \right\rbrace= \left[0, +\infty \right]$, we obtain, for all $x \geq 0$ \
\begin{equation*}
f(x)-g(x)=\left\{
	\begin{aligned}
		0,~&~\text{if } x> 0,\\
		-1,~&~\text{if } x=0,\\
		+\infty, ~&~ \text{otherwise,}
	\end{aligned}
	\right.
\end{equation*}
and $v\eqref{eq:Primal_problem}=\inf\left\lbrace f(x)-g(x):x \geq 0 \right\rbrace=-1$.\\ 
On the other hand, we have, for $(u^*, v^*, \alpha) \in \R^3$,  
\begin{align*}
g^c(u^*,v^*,\alpha) &= \sup_\R \left\lbrace c(x,(u^*, v^*, \alpha))-g(x)\right\rbrace\\
&=\left\{
	\begin{aligned}
		\sup\left\lbrace \sup_{x \geq 0} \left\lbrace x(u^*-1) \right\rbrace,-1 \right\rbrace,~&~\text{if } v^* \leq 0, \alpha >0\\
		+\infty, ~&~ \text{otherwise,}
	\end{aligned}
	\right.	\\
	&=\left\{
	\begin{aligned}
		0,~&~\text{if } u^* \leq 1, v^* \leq 0, \alpha >0\\
		+\infty, ~&~ \text{otherwise.}
	\end{aligned}
	\right.
\end{align*}
Moreover, $\dom g^c=\left\lbrace (u^*, v^*, \alpha)\in \R^3:u^* \leq 1, v^* \leq 0, \alpha >0 \right\rbrace$.
Now, for all $(u^*, v^*, \alpha)\in \dom g^c$ and $\lambda \geq 0$,
\begin{align*}
(f+\lambda h)^c(u^*,v^*,\alpha) &= \sup_\R \left\lbrace c(x,(u^*, v^*, \alpha))-(f+\lambda h)(x)\right\rbrace\\
&=\sup_{x\geq 0} \left\lbrace x(u^*-1+\lambda)\right\rbrace
=\left\{
	\begin{aligned}
		0,~&~\text{if } u^* \leq 1-\lambda\\
		+\infty, ~&~ \text{otherwise.}
	\end{aligned}
	\right.	\\
	\end{align*}
Finally, we have, for all $(u^*, v^*, \alpha)\in \dom g^c$ and $\lambda \geq 0$,
\begin{equation*}
g^c(u^*,v^*,\alpha)-(f+\lambda h)^c(u^*,v^*,\alpha)=
\left\{
	\begin{aligned}
		0,~&~\text{if }u^* \leq 1-\lambda\\
		-\infty, ~&~ \text{otherwise,}
	\end{aligned}
	\right.	\\
\end{equation*}	
and, for all $\lambda > 0$,$$\inf_{\dom g^c}\left\lbrace g^c(u^*,v^*,\alpha)-(f+\lambda h)^c(u^*,v^*,\alpha)\right\}=-\infty,$$
but, for $\lambda =0$,$$\inf_{\dom g^c}\left\lbrace g^c(u^*,v^*,\alpha)-f^c(u^*,v^*,\alpha)\right\}=0,
$$ hence $v\eqref{eq:Dual_problem_2}=0$ and weak duality does not hold for \eqref{eq:Primal_problem}$-$\eqref{eq:Dual_problem_2}. In addition, it is not difficult to compute the optimal value of \eqref{eq:Dual_problem_1}, since 
\begin{equation*}
	 f(x)-g(x)+\lambda h(x)=
	\left\{
	\begin{aligned}
		-\lambda x,~&~\text{if }x >0\\
		-1,~&~ \text{if }x=0\\
		+\infty, ~&~ \text{otherwise,}
	\end{aligned}
	\right.	\\
\end{equation*}
and, for $\lambda \geq 0$,
\begin{equation*}
	 \inf_{x\geq 0}\left\lbrace  f(x)-g(x)+\lambda h(x) \right\rbrace=
	\left\{
	\begin{aligned}
		-\infty,~&~\text{if }\lambda >0\\
		-1,~&~ \text{if }\lambda=0,\\
	\end{aligned}
	\right.	\\
\end{equation*}
then
\begin{equation*}
	v\eqref{eq:Dual_problem_1}=\sup_{\lambda \geq 0}~\inf_{x\geq 0} \left\lbrace  f(x)-g(x)+\lambda h(x) \right\rbrace=-1
\end{equation*}
and, finally, 
\begin{equation*}
	v\eqref{eq:Dual_problem_1}= v\eqref{eq:Primal_problem} < v\eqref{eq:Dual_problem_2}. 
\end{equation*}
\begin{flushright}
	$\square$
\end{flushright}
\end{example}
As we can observe in the previous example, not only weak duality between $(P)$ and $\eqref{eq:Dual_problem_1}$ is guaranteed, but also there is no duality gap, condition that may not happen, in general. Moreover, in duality theory, there exists an interest for finding conditions which assure strong duality, i.e., when there is no duality gap and the dual problem is solvable.
 
Therefore, the rest of the paper is devoted to the derivation of conditions ensuring weak duality, in case $g$ is not e-convex, for \eqref{eq:Primal_problem}$-$\eqref{eq:Dual_problem_2}, and conditions for zero duality gap and strong duality for both pairs \eqref{eq:Primal_problem}$-$\eqref{eq:Dual_problem_2} and \eqref{eq:Primal_problem}$-$\eqref{eq:Dual_problem_1}. We will also compare both strong dualities making use of the theoretical results we derive in Section \ref{sec:Cond_WD}.

\section{Duality results for \eqref{eq:Primal_problem}$-$\eqref{eq:Dual_problem_2}}
\label{sec:Cond_WD}
As we have seen in Section \ref{sec:Dual_problems}, if $g$ is an e-convex function, there exists weak duality between \eqref{eq:Primal_problem} and \eqref{eq:Dual_problem_2}. If we eliminate the even convexity assumption on $g$, we can also guarantee weak duality under the hypothesis of the following proposition.

\begin{proposition}
\label{prop:Prop_WD}
Let us suppose that $g$ in the primal problem $(P)$ has a proper e-convex minorant. If \eqref{eq:Primal_problem} is solvable and there exists a solution where $g$ is e-convex, then weak duality holds for \eqref{eq:Primal_problem}$-$\eqref{eq:Dual_problem_2}.
\end{proposition}
\begin{proof}
We build the problem
\begin{equation}
	\label{eq:Primal_problem_prop}
	\tag{$P_e$}
	\inf_A \left\lbrace f(x)-\eco g(x)\right\rbrace.
	\end{equation}

According to \cite{FVR2016} again, we can derive a dual problem for $(P_e)$, verifying weak duality

\begin{equation}
	\label{eq:Dual_problem_prop}
	\tag{$D_{e,L}$}
	\sup_{\RTD_+}\inf_{X} \left\{f(x)-\eco g(x)+\lambda h(x)\right\}.
	\end{equation}

In this case, $\eco g$ is e-convex, and following the steps done for the construction of  $\eqref{eq:Dual_problem_2}$ for $(P)$, we obtain an equivalent problem to $(D_{e,L})$
\begin{equation*}
	\sup_{\RTD_+}\inf_{ W} \left\{(\eco g)^c(u^*,v^*,\gamma)-(f+\lambda h)^c(u^*,v^*,\gamma)\right\}.
\end{equation*}
However, it is easy to see, by means of Theorem \ref{thm:Theorem_charac} $(ii)$, that  $(\eco g)^{c}=(g^{cc^\prime})^c=g^c$, so  $\eqref{eq:Dual_problem_2}$ and $(D_{e,L})$ are equivalent, thus $v\eqref{eq:Dual_problem_2}\leq v(P_e)$. Naming $x_0$ to a solution of $(P)$ where $g$ is e-convex, $g(x_0)=\eco g(x_0)$, and we obtain
\begin{align*}
	v(P)&=f(x_0)-g(x_0)+\delta_A(x_0)= f(x_0)-\eco g(x_0) + \delta_A(x_0)\\
	&\geq v(P_e)\geq v\eqref{eq:Dual_problem_2}.
\end{align*}
\end{proof}

\begin{corollary}
\label{cor:WD}
If $(P)$ is solvable and $x_0$ is a solution verifying $\partial_c g(x_0)\neq \emptyset$, then weak duality holds for \eqref{eq:Primal_problem}$-$\eqref{eq:Dual_problem_2}.
\end{corollary}
\begin{proof}
This result is an immediate consequence of \cite[Th.\,3 (ii)]{FV2022}.
\end{proof}

In order to derive additional conditions ensuring weak and strong dualities, and zero duality gap for this dual pair, we define the following sets, being $A$ the set given in \eqref{eq:Set_A}
\begin{align}
	\label{eq:Set_Kp}
	K &:= \bigcup_{\RTD_+}\bigcap_{\dom g^c} \left\{ \epi(f+\lambda h)^c-(u^*,v^*,\gamma,g^c(u^*,v^*,\gamma))\right\},\\
	\label{eq:Set_Lambda}
	\Lambda &:= \bigcap_{\dom g^c} \left\{ \epi(f+\delta_A)^c-(u^*,0,0,g^c(u^*,v^*,\gamma))\right\}.
\end{align}
Regarding these sets, we state the following lemma.
\begin{lemma}
\label{lem:Lemma_aux_C}
Let $K$ and $\Lambda$ the sets defined in \eqref{eq:Set_Kp} and \eqref{eq:Set_Lambda}, respectively. It holds
\begin{itemize}
	\item[i)] $\Lambda=\epi(f-\eco g+\delta_A)^c$.
	\item[ii)] $K \supseteq \bigcup_{\RTD_+} \epi(f-\eco g +\lambda h)^c$.	
\end{itemize}
\end{lemma}
\begin{proof}
 \textit{i)} The proof is stated in \cite[Prop. 5.5]{FV2023}.\\
 \textit{ii)} For any $\lambda\in\RTD_+$, it is not difficult to see that
\begin{equation*}
	f-\eco g+\lambda h = \inf_{\dom g^c}\left\{ -c(\cdot,(u^*,v^*,\gamma))+g^c(u^*,v^*,\gamma)+f+\lambda h\right\}.
\end{equation*}
Applying $c$-conjugation standard properties,  we obtain
\begin{align*}
	(f-&\eco g+\lambda h)^c(x^*,y^*,\alpha) = \sup_{X} \left\{ c(x,(x^*,y^*,\alpha))-(f-\eco g+\lambda h)(x)\right\}\\
	&=\sup_{X,~ \dom g^c} \left\{ c(x,(x^*,y^*,\alpha))+c(x,(u^*,v^*,\gamma))-(f+\lambda h)(x)-g^c(u^*,v^*,\gamma)\right\}\\
	&\geq \sup_{X,~\dom g^c}\left\{ c(x,(x^*+u^*,y^*+v^*,\alpha+\gamma))-(f+\lambda h)(x)-g^c(u^*,v^*,\gamma)\right\}\\
	&\geq (f+\lambda h)^c(x^*+u^*,y^*+v^*,\alpha+\gamma)-g^c(u^*,v^*,\gamma),
\end{align*}
for all $(u^*,v^*,\gamma)\in\dom g^c$. Then, if $(x^*,y^*,\alpha,\beta)\in\epi(f-\eco g+\lambda h)^c$, 
	$$(x^*+u^*,y^*+v^*,\alpha+\gamma,\beta+g^c(u^*,v^*,\gamma))\in\epi(f+\lambda h)^c$$
for all $(u^*,v^*,\gamma)\in\dom g^c$. Since $\lambda$ was any sequence from $\RTD_+$, we get the inclusion given in $ii)$.
\end{proof}
\noindent
Now, we continue presenting a useful lemma from \cite{FV2023} and adapting \cite[Lem.~4.1~(ii)]{FV2023} to the set $K$ defined in \eqref{eq:Set_Kp}.
\begin{lemma}\normalfont{\cite[Lem.~4.1~(i)]{FV2023}}
\label{lem:Lemma_FV23}
Let $\beta\in\R$ and $A$ the set defined in \eqref{eq:Set_A}. Then, there exists $\delta>0$ such that $(0,0,\delta,\beta)\in\epi(f-g+\delta_A)^c$ if and only if $v\eqref{eq:Primal_problem}\geq -\beta$.
\end{lemma}

\begin{remark}
\label{R1}
Taking into account that $(f-g+\delta_A)^c(0,0,\delta)=\sup_X \left\lbrace-(f-g+\delta_A)(x) \right\rbrace$, for any $\delta >0$, it is evident that $(0,0,\delta,\beta)\in\epi(f-g+\delta_A)^c$, for a certain $\delta >0$, if and only if $(0,0,\delta,\beta)\in\epi(f-g+\delta_A)^c$, for all $\delta >0$. 
\end{remark}
\begin{lemma}
\label{lem:Lemma_aux}
Let $K$ be the set defined in \eqref{eq:Set_Kp} and $\beta\in\R$. Then, $(0,0,0,\beta)\in K$ if and only if there exists $\lambda\in\RTD_+$ such that for all $(u^*,v^*,\gamma)\in\dom g^c$, it holds
\begin{equation}
\label{eq:Lemma_aux_WD_0}
	g^c(u^*,v^*,\alpha)-(f+\lambda h)^c(u^*,v^*,\gamma) \geq -\beta.
\end{equation}
\end{lemma}
\begin{proof}
 A point $(0,0,0,\beta)\in K$ if and only if there exists $\lambda\in\RTD_+$ such that, for all $(u^*,v^*,\gamma)\in\dom g^c$, 
\begin{equation*}
	(0,0,0,\beta)+(u^*,v^*,\gamma,g^c(u^*,v^*,\gamma))\in\epi(f+\lambda h)^c,
\end{equation*}
which is equivalent to the inequality
\begin{equation*}
	\left( f+\lambda h\right)^c(u^*,v^*,\gamma)\leq \beta +g^c(u^*,v^*,\gamma)
\end{equation*}
and \eqref{eq:Lemma_aux_WD_0} fulfills.
\end{proof}

Duality results for the primal-dual pair \eqref{eq:Primal_problem}$-$\eqref{eq:Dual_problem_2} are based on the set
\begin{equation}\label{eq:Set_Kpp}
	K' :=\left\{ (0,0,\delta,\beta)\,:\,(0,0,0,\beta)\in K\,\text{ and } \, \delta>0\right\} 
	\end{equation}
and the auxiliary set
\begin{equation}\label{eq:Set_B}
	B:=\left\{(0,0)\times\R_{++}\times\R\right\}\subseteq W\times \R. 
\end{equation}
Together with the sets $K^\prime$ and $B$, the following lemmas will be necessary in the sequel.

\begin{lemma}
\label{lem:epi(f-g+dA}
If $v(P) \in \R$, it holds
$$\epi(f-g+\delta_A)^c \cap B=\left\{ (0,0,\delta,\beta)\,:\delta >0,\beta\geq-v\eqref{eq:Primal_problem} \right\}=\bigcap_{X}\epi c^\prime(\cdot,x)-v\eqref{eq:Primal_problem}.$$
Consequently, $\epi(f-g+\delta_A)^c \cap B$ is an $\ep$-convex set.
\end{lemma}
\begin{proof}
The first equality derives directly from Lemma \ref{lem:Lemma_FV23}.
Now a point $(x^*, y^*, \delta, \beta)$ belongs to the set $\bigcap_{X}\epi c^\prime(\cdot,x)-v\eqref{eq:Primal_problem}$ if and only if, for all $x \in X$,
$$c'((x^*, y^*, \delta),x) -v(P) \leq \beta.$$
Equivalently, for all $x \in X$,
$$\left\langle x,y^* \right\rangle < \delta \text{ and } \left\langle x,x^* \right\rangle -v(P) \leq \beta.$$
This means that, necessarily, $x^*=y^*=0$, $\delta >0$ and $ -v(P) \leq \beta.$
So, $\epi(f-g+\delta_A)^c \cap B$ is the epigraph of the $\ep$-convex function $\sup_X \left\lbrace c^\prime(\cdot,x)-v\eqref{eq:Primal_problem}\right\rbrace$ and, hence, it is $\ep$-convex.
\end{proof}

\begin{lemma}
\label{lem: kprime between}
Let us asumme that $v\eqref{eq:Dual_problem_2} \in  \R$. Then
\begin{equation}
\label{eq:incl_kprime}
\left\{ (0,0,\delta,\beta)\,:\delta >0,\beta>-v\eqref{eq:Dual_problem_2}\right\} \subseteq K' \subseteq \left\{ (0,0,\delta,\beta)\,:\delta >0,\beta\geq -v\eqref{eq:Dual_problem_2}\right\}.
\end{equation}
In particular, if $\eqref{eq:Dual_problem_2}$ is solvable, it holds
\begin{equation}
\label{eq:equa_kprime}
K' = \left\{ (0,0,\delta,\beta)\,:\delta >0,\beta\geq -v\eqref{eq:Dual_problem_2}\right\}
\end{equation}
and, consequently, $K'$ is $\ep$-convex.
\end{lemma}
\begin{proof}
Take a point $(0,0,\delta,\beta)$ in $W \times \R$ with $\delta >0$ and $\beta>-v\eqref{eq:Dual_problem_2}$. Since $v\eqref{eq:Dual_problem_2}> -\beta$, there would exist $\lambda \in \RTD_+$ such that, for all $(u^*,v^*,\gamma)\in\dom g^c$,
\begin{equation*}
	g^c(u^*,v^*,\gamma)-(f+\lambda h)^c(u^*,v^*,\gamma)\geq -\beta,
\end{equation*}
and, according to Lemma \ref{lem:Lemma_aux}, $(0,0,0, \beta) \in K$ and therefore $(0,0,\delta,\beta) \in K'.$ 
The second inclusion in \eqref{eq:incl_kprime}  also comes directly from Lemma  \ref{lem:Lemma_aux}.\\
Now, let us suppose that $v\eqref{eq:Dual_problem_2}$ is attainable at $\bar \lambda$. If $\beta\geq -v\eqref{eq:Dual_problem_2}$, we have
$$g^c(u^*,v^*,\gamma)-(f+\bar \lambda h)^c(u^*,v^*,\gamma)\geq -\beta$$
for all $(u^*,v^*,\gamma)\in\dom g^c$ and, according to Lemma \ref{lem:Lemma_aux}, $(0,0,0,\beta) \in K$ hence, for all $\delta >0$, $(0,0,\delta,\beta) \in K'$ and  \eqref{eq:equa_kprime} holds.
It is easy to see, as in the proof of the previous lemma, that, 
$$\left\{ (0,0,\delta,\beta)\,:\delta >0,\beta\geq -v\eqref{eq:Dual_problem_2}\right\}= \bigcap_{X}\epi c^\prime(\cdot,x)-v\eqref{eq:Dual_problem_2}.$$
\end{proof}


\begin{proposition}
\label{prop:WD_DL}
Let $A$, $K'$ and $B$ be the sets defined in \eqref{eq:Set_A}, \eqref{eq:Set_Kpp} and \eqref{eq:Set_B}, respectively. Then weak duality holds for \eqref{eq:Primal_problem}$-$\eqref{eq:Dual_problem_2} if and only if
\begin{equation}
	\label{eq:WD_DL_Conte}
	K'\subseteq \epi(f-g+\delta_A)^c\cap B.
\end{equation}
\end{proposition}
\begin{proof}
Let us assume that $v\eqref{eq:Primal_problem}\geq v\eqref{eq:Dual_problem_2}.$ If $v\eqref{eq:Primal_problem}=-\infty$, then by Lemma \ref{lem:Lemma_FV23}, $\epi(f-g+\delta_A)^c\cap B =\emptyset$. On the other hand, because $v\eqref{eq:Dual_problem_2}=-\infty$ too, by Lemma \ref{lem:Lemma_aux}, $K'=\emptyset$.\\
Now, suppose that $v\eqref{eq:Primal_problem} \in \R$. Weak duality implies that also $v\eqref{eq:Dual_problem_2} \in \R$ and  the inclusion \eqref{eq:WD_DL_Conte} is a direct consequence from Lemmas  \ref{lem: kprime between}    and \ref{lem:epi(f-g+dA}.\\
Now we assume that \eqref{eq:WD_DL_Conte} holds. If $v\eqref{eq:Primal_problem}<v\eqref{eq:Dual_problem_2}$, we can take $\beta\in\R$ in such a way that
\begin{equation}
	\label{eq:Chain_Prop_WD}
	v\eqref{eq:Primal_problem}<-\beta<v\eqref{eq:Dual_problem_2}.
\end{equation}
Since $\beta>-v\eqref{eq:Dual_problem_2}$, by Lemma \ref{lem: kprime between}, for all $\delta >0$, $(0,0,\delta, \beta) \in K'$ and therefore, by \eqref{eq:WD_DL_Conte}, $(0,0,\delta, \beta) \in \epi(f-g+\delta_A)^c\cap B$, which is not possible, in virtue of Lemma \ref{lem:epi(f-g+dA}.
\end{proof}

\begin{proposition}
\label{prop:Zero_DG}
Let $A$, $K'$ and $B$ be the sets defined in \eqref{eq:Set_A}, \eqref{eq:Set_Kpp} and \eqref{eq:Set_B}, respectively. Then there exists zero duality gap for \eqref{eq:Primal_problem}$-$\eqref{eq:Dual_problem_2} if and only if
\begin{equation}
	\label{eq:Assertion_Prop_ZDG}
	\epco K' = \epi(f-g+\delta_A)^c\cap B.
\end{equation}
\end{proposition}
\begin{proof}
Suppose that $v\eqref{eq:Primal_problem}=v\eqref{eq:Dual_problem_2} \in \R$ (otherwise,  $\epi(f-g+\delta_A)^c\cap B =K'=\emptyset$ and $\emptyset$ is $\ep$-convex, since it is the epigraph of the $\ep$-convex function $S= \sup_{X \times \R}c'(\cdot,x)-\alpha)$.\\
In first place, we will see that
\begin{equation}
\label{eq:epcoequality}
\epco \left\{ (0,0,\delta,\beta)\,:\delta >0,\beta>-v\eqref{eq:Primal_problem}\right\}=\left\{ (0,0,\delta,\beta)\,:\delta >0,\beta\geq -v\eqref{eq:Primal_problem}\right\}.
\end{equation}
Since, according to Lemma \ref{lem:epi(f-g+dA}, the set $\left\{ (0,0,\delta,\beta)\,:\delta >0,\beta\geq -v\eqref{eq:Primal_problem}\right\}$ is $\ep$-convex, it is clear that
$$\epco \left\{ (0,0,\delta,\beta)\,:\delta >0,\beta>-v\eqref{eq:Primal_problem}\right\}\subseteq \left\{ (0,0,\delta,\beta)\,:\delta >0,\beta\geq -v\eqref{eq:Primal_problem}\right\}.$$
Let us denote by $H:W \rightarrow \R$ the $\ep$-convex function whose epigraph is $\epco \left\{ (0,0,\delta,\beta)\,:\delta >0,\beta>-v\eqref{eq:Primal_problem} \right\rbrace$. Let us assume that
$$H=\sup_D c' (\cdot, x)-\alpha,$$ where $D\subseteq X \times \R$.
Then, for all $\delta >0$ and $\beta>-v\eqref{eq:Primal_problem}$, for all $(x,\alpha) \in D$, we have
\begin{equation*}
	c^\prime((0,0,\delta),x)-\alpha \leq \beta,
\end{equation*}
or, equivalently, 
\begin{equation}
	\label{eq:Contradiction_ZDG}
	-\alpha \leq \beta,
\end{equation}
for all $\beta>-v\eqref{eq:Primal_problem}$, which means that $\alpha \geq v\eqref{eq:Primal_problem}$, and $D$ is any nonempty subset of $X \times \left[v(P), + \infty \right[.$
Let us show that
$$ \left\{ (0,0,\delta,\beta)\,:\delta >0,\beta\geq -v\eqref{eq:Primal_problem}\right\}\subseteq\epco \left\{ (0,0,\delta,\beta)\,:\delta >0,\beta>-v\eqref{eq:Primal_problem}\right\}.$$
For all $(x,\alpha) \in D$, if $\delta >0$ and $\beta \geq -v(P)$,
$$c'((0,0,\delta),x)-\alpha=-\alpha \leq -v(P)\leq \beta$$
and hence $(0,0,\delta,\beta) \in \epi c'(\cdot, x)-\alpha$, therefore
$$ \left\{ (0,0,\delta,\beta)\,:\delta >0,\beta\geq -v\eqref{eq:Primal_problem}\right\}\subseteq \bigcap_D \epi c'(\cdot, x)-\alpha=\epi H$$
and \eqref{eq:epcoequality} fulfills.\\
Now, let us prove equality \eqref{eq:Assertion_Prop_ZDG} under the hypothesis of zero duality gap. According to Lemma \ref{lem: kprime between}, we will have
$$\left\{ (0,0,\delta,\beta)\,:\delta >0,\beta>-v\eqref{eq:Primal_problem}\right\} \subseteq K' \subseteq \left\{ (0,0,\delta,\beta)\,:\delta >0,\beta\geq -v\eqref{eq:Primal_problem}\right\}.$$
In virtue of Lemma \ref{lem:epi(f-g+dA} and the definition of e$^\prime$-convex hull, we will have
$$\epco \left\{ (0,0,\delta,\beta)\,:\delta >0,\beta>-v\eqref{eq:Primal_problem}\right\} \subseteq \epco K' \subseteq \epi(f-g+\delta_A)^c\cap B.$$
According to equality \eqref{eq:epcoequality}, we obtain \eqref{eq:Assertion_Prop_ZDG}.\\ 
Now, let us suppose that \eqref{eq:Assertion_Prop_ZDG} is true. Applying Proposition \ref{prop:WD_DL}, we have $v\eqref{eq:Primal_problem}\geq v\eqref{eq:Dual_problem_2}$. If $v(P)=-\infty$, necessarily $v\eqref{eq:Dual_problem_2}=-\infty$, then suppose that $v\eqref{eq:Primal_problem}\in \R$. Therefore $v\eqref{eq:Dual_problem_2} \in \R$ too.
Now, by Lemma \ref{lem:epi(f-g+dA} and \eqref{eq:Assertion_Prop_ZDG}, we have
$$\epco K' = \left\{ (0,0,\delta,\beta)\,:\delta >0,\beta\geq -v\eqref{eq:Primal_problem}\right \},$$
and by Lemma \ref{lem: kprime between}, taking e'convex hull, we obtain
$$\epco K' \subseteq \left\{(0,0,\delta,\beta): \delta>0,\, \beta\geq -v\eqref{eq:Dual_problem_2}\right\},$$
hence
\begin{equation}
\label{eq:Chain_prop_ZDG_RtoL}
\left\{(0,0,\delta,\beta): \delta>0,\, \beta\geq -v\eqref{eq:Primal_problem}\right\} \subseteq \left\{(0,0,\delta,\beta): \delta>0,\, \beta\geq -v\eqref{eq:Dual_problem_2}\right\}
\end{equation} 
and $v(P)\leq v\eqref{eq:Dual_problem_2}$. 
\end{proof}

\begin{proposition}
\label{prop:SD_DL}
Let $A$, $K'$ and $B$ be the sets defined in \eqref{eq:Set_A}, \eqref{eq:Set_Kpp} and \eqref{eq:Set_B}, respectively. Then strong duality holds for \eqref{eq:Primal_problem}$-$\eqref{eq:Dual_problem_2} if and only if
\begin{equation}
\label{eq:SD}
	K' = \epi(f-g+\delta_A)^c\cap B.
\end{equation}
\end{proposition}
\begin{proof}
Let us assume that there exists strong duality between $(P)$ and \eqref{eq:Dual_problem_2} and both optimal values are finite (the infinite case is discussed in Proposition \ref{prop:Zero_DG}). Acoording to Lemma \ref{lem: kprime between}, $K'$ is $\ep$-convex, and from Proposition \ref{prop:Zero_DG}, we obtain \eqref{eq:SD}.\\
Let us check the converse implication. If $K' = \epi(f-g+\delta_A)^c\cap B$, it implies in particular that $K'$ is an e$^\prime$-convex set, then by  Proposition \ref{prop:Zero_DG}, $v\eqref{eq:Dual_problem_2}= v(P)$. In case $v(P)=-\infty$, strong duality holds trivially, so let us assume that $v(P)=-\beta\in\R$. \\
By Lemma \ref{lem:Lemma_FV23}, $(0,0,\delta,\beta) \in \epi(f-g+\delta_A)^c \cap B$, for any $\delta >0$, and due to the equality in the hypothesis, $(0,0,0,\beta) \in K$. Then according to Lemma \ref{lem:Lemma_aux}, there exists $\lambda_0\in\RTD_+$ such that for all $(u^*,v^*,\gamma)\in\dom g^c$,
\begin{equation*}
	v\eqref{eq:Dual_problem_2} \geq g^c(u^*,v^*,\gamma)-(f+\lambda_0 h)^c(u^*,v^*,\gamma)\geq -\beta=v\eqref{eq:Dual_problem_2},
\end{equation*}
i.e., \ref{eq:Dual_problem_2} is solvable and there exists strong duality for the dual pair $(P)-$(\ref{eq:Dual_problem_2}).
\end{proof}

\section{Zero duality Gap and Strong duality for \eqref{eq:Primal_problem}$-$\eqref{eq:Dual_problem_1}}
\label{sec:Comp_DL1}
Characterizations for zero duality gap and strong duality for \eqref{eq:Primal_problem}$-$\eqref{eq:Dual_problem_1} can be obtained by means of the sets 
\begin{align}
	\Omega&:= \bigcup_{\RTD_+}\bigcap_{\dom g} \left[ (g-f-\lambda h)(x), +\infty\right[ \subseteq \R \nonumber,\\
		\label{eq:Set Omegapp}
	\Omega'&:=\left\{ (0,0,\delta,\beta)\,:\,\delta >0, \beta\in \Omega\right\}\subseteq W \times \R,
\end{align}
with similar results to those in Section \ref{sec:Cond_WD} for \eqref{eq:Primal_problem}$-$\eqref{eq:Dual_problem_2}. The following proposition has analogous proofs than Propositions \ref{prop:Zero_DG} and \ref{prop:SD_DL}. The proof of $i)$ comes directly from the fulfilment of weak duality for $(P)-(D_L)$.

\begin{proposition}
\label{prop:Zero_DG2}
Let $A$, $\Omega'$ and $B$ be the sets defined in \eqref{eq:Set_A}, \eqref{eq:Set Omegapp} and \eqref{eq:Set_B}, respectively. The following statements hold
\begin{itemize}
	\item[i)] $\Omega'\subseteq \epi(f+g-\delta_A)^c \cap B$.
	\item[ii)] There exists zero duality gap for $(P)-(D_L)$ if and only if 
	$$\epco \Omega' = \epi(f-g+\delta_A)^c\cap B.$$
	\item[iii)] There exists strong duality for $(P)-(D_L)$ if and only if 
	$$\Omega'
	= \epi(f-g+\delta_A)^c\cap B.$$
\end{itemize}
\end{proposition}

\subsection{Relation between strong dualities for (P)$-$\eqref{eq:Dual_problem_1} and (P)$-$\eqref{eq:Dual_problem_2}}
This subsection aims to obtain conditions under which both strong dualities are equivalent. The following theorem can be viewed as Lagrange counterpart of Theorem 5.6 from \cite{FV2023}.

\begin{theorem}
\label{thm:SD_mixed}
Let $B$, $\Omega'$, $K'$ and $\Lambda$ the sets defined in \eqref{eq:Set_B}, \eqref{eq:Set Omegapp}, \eqref{eq:Set_Kpp} and \eqref{eq:Set_Lambda}, respectively, and assume that
\begin{equation} \label{eq:problemPe}
	\epi(f-g+\delta_A)^c\cap B = \Lambda\cap B
\end{equation}
with $g$ having a proper e-convex minorant.
 Then the following statements are equivalent
\begin{itemize}
	\item[i)] There exists strong duality for $(P)-(D_L)$;
	\item[ii)] There exists strong duality for $(P)-\left (\Db_L \right )$ and $\Omega'=K'$.
\end{itemize}
\end{theorem}
\begin{proof}
The implication $ii) \Rightarrow i)$ is a direct consequence from Propositions \ref{prop:SD_DL} and \ref{prop:Zero_DG2}. In this implication, condition \eqref{eq:problemPe} is not necessary.\\
$i) \Rightarrow ii)$ First of all, recall the dual problems \eqref{eq:Primal_problem_prop} and \eqref{eq:Dual_problem_prop} defined in the proof of Proposition \ref{prop:Prop_WD}
\begin{equation}
	\label{eq:Primal_problem_prop}
	\tag{$P_e$}
	\inf_A \left\lbrace f(x)-\eco g(x)\right\rbrace
	\end{equation}
\begin{equation}
	\label{eq:Dual_problem_prop}
	\tag{$D_{e,L}$}
	\sup_{\RTD_+}\inf_{X} \left\{f(x)-\eco g(x)+\lambda h(x)\right\}.
	\end{equation}
In that proposition it was shown that $v(P_e) \geq v(\Db_L)$.\\
Now, we will prove that $v(P)=v(P_e)$. Since $\eco g\leq g$, then $v(P_e)\geq v(P)$. In case $v(P_e)>v(P)$, it is enough to consider any $\beta\in\R$ such that
\begin{equation*}
	v(P)<-\beta\leq v(P_e).
\end{equation*}
Applying Lemma \ref{lem:Lemma_FV23} to $(P_e)$, we obtain that $(0,0,\delta,\beta)\in\epi(f-\eco g+\delta_A)^c$, for certain $\delta>0$, and by Lemma \ref{lem:Lemma_aux_C} $i)$, $(0,0,\delta,\beta)\in\Lambda$. In virtue of equality \eqref{eq:problemPe}, $(0,0,\delta,\beta)\in\epi(f-g+\delta_A)^c$, so again, applying Lemma \ref{lem:Lemma_FV23} to $(P)$, $v(P)\geq-\beta$, which is a contradiction. Then, $v(P)=v(P_e)$. This means that weak duality holds for $(P)-(\Db_L)$. Equivalently, according to Propositions
\ref{prop:WD_DL} and \ref{prop:Zero_DG2},
\begin{equation*}
	\label{eq:WD_DL_Cont}
	K'\subseteq \epi(f-g+\delta_A)^c\cap B =\Omega'.
\end{equation*}
Finally, we will see that $\Omega'\subseteq K'$. Take any point $(0,0,\delta,\beta)\in\Omega'$. We shall prove that $(0,0,0,\beta)\in K$. By definition of the set $\Omega'$ in \eqref{eq:Set Omegapp}, there exists $\lambda \in \RTD_+$ such that for all $x \in \dom g$,
\begin{equation*}
	-g(x)+\beta \geq -(f+\lambda h)(x).
\end{equation*}
Hence, for all $(u^*, v^*, \gamma) \in W$,
\begin{equation*}
	g^c(u^*, v^*, \gamma)+\beta \geq (f+\lambda h)^c(u^*, v^*, \gamma),
\end{equation*}
and $(u^*, v^*, \gamma,g^c(u^*, v^*, \gamma)+\beta)\in \epi(f+\lambda h)^c$, in particular for all $(u^*, v^*, \gamma) \in \dom g^c$, concluding that
$(0,0,0,\beta)\in K$.
Hence, $K'= \epi(f-g+\delta_A)^c\cap B =\Omega'$ and, moreover, according to Proposition \ref{prop:SD_DL}, there exists strong duality for $(P)-\left (\Db_L \right )$.
\end{proof}

\section{Fenchel-Lagrange duality}
\label{sec:FL_duality}
Having studied characterizations for Fenchel and Lagrange weak and strong dualities, and also zero duality gap for DC primal optimization problems, one cannot forget the natural combination between Fenchel and Lagrange dual problems. This mixed model, which is called \textit{Fenchel-Lagrange} dual problem, was introduced by Bot and Wanka for a convex optimization primal problem in \cite{BW2002} by means of Fenchel conjugation scheme and making use of the pertubational approach. In \cite{FV2017} it was built an alternative Fenchel-Lagrange dual using the $c$-conjugation pattern, fulfilling weak duality by construction, which reads, in the DC optimization context, like this
\begin{equation*}
	\label{eq:FL_Dold}
	\tag{$D_{FL}$}
	\sup_{\substack{\lambda\in\RTD_+,\\x^*,y^*\in X^*,\\ \alpha_1+\alpha_2>0}} \left\{ -(f-g)^c(-x^*,-y^*,\alpha_1)-(\lambda h)^c(x^*,y^*,\alpha_2)\right\}.
\end{equation*}
Applying \cite[Sec.3]{FV2023} to the previous Fenchel-Lagrange dual problem, we can consider the following reformulation of \eqref{eq:FL_Dold},
  \begin{equation*}
	\label{eq:FL_Dold}
	\tag{$D_{FL}$}
	\sup_{\RTD_+, Z} \left\{ -(f-g)^c(-x^*,-y^*,\alpha)-(\lambda h)^c(x^*,y^*,\alpha)\right\},
\end{equation*}
where $Z:=X^* \times X^*\times \R_{++}$.
Moreover, as it happened for Fenchel and Lagrange dual problems, it is also possible to derive an alternative formulation for \eqref{eq:FL_Dold}, provided that the function $g$ is e-convex. Proceeding along the lines of \cite[Sec.3]{FV2023} and replacing the function $g$ by $g^{cc^{\prime}}$ in $(f-g)^c$, we get

\begin{equation}
	\label{eq:final_form_f_g}
	(f-g)^c(-x^*,-y^*,\alpha)
	=
	\sup_{\dom g^c} \left\{f^c(u^*-x^*,-y^*,\alpha)-g^c(u^*,v^*,\gamma)\right\}.
\end{equation}
Applying \eqref{eq:final_form_f_g} in \eqref{eq:FL_Dold} we obtain an equivalent refomulation, in case $g$ is e-convex,
\begin{equation}
	\label{eq:FL_Dnew}
	\tag{$\overline{D}_{FL}$}
	\sup_{\substack{\RTD_+,Z}}\biggl\{ \inf_{\dom {g^c}} \left\{ g^c(u^*,v^*,\gamma)-f^c(u^*-x^*,-y^*,\alpha)\right\} - (\lambda h)^c(x^*,y^*,\alpha)\biggr\}.
\end{equation}
Considering \eqref{eq:FL_Dnew} as a new Fenchel-Lagrange dual problem for \eqref{eq:Primal_problem} with $g$ just a convex, but not necessarily e-convex function, even weak duality might fail, as one can see in the following example.

\begin{example}
Let $T$ be a singleton and $f,g,h:X\to\Ramp$ defined as in Example \ref{ex:No_equivalence_L}
\begin{equation*}
	f(x)=\left\{
	\begin{aligned}
		x,~&~\text{if } x\geq 0,\\
		+\infty, ~&~ \text{otherwise;}
	\end{aligned}
	\right.
	\hspace{0.35cm}
	g(x)=\left\{
	\begin{aligned}
		x,~&~\text{if } x> 0,\\
		1,~&~\text{if } x=0,\\
		+\infty, ~&~ \text{otherwise;}
	\end{aligned}
	\right.
	\hspace{0.35cm}
	h(x)=\left\{
	\begin{aligned}
		-x,~&~\text{if } x\geq 0,\\
		+\infty, ~&~ \text{otherwise.}
	\end{aligned}
	\right.
\end{equation*}
We obtained $A=\left[0, +\infty \right]$ and $v\eqref{eq:Primal_problem}=-1$.\\ 
Moreover, $\dom g^c=\left\lbrace (u^*, v^*, \alpha)\in \R^3:u^* \leq 1, v^* \leq 0, \alpha >0 \right\rbrace$ and, for all $(u^*, v^*, \alpha)\in \dom g^c$,
\begin{equation*}
g^c(u^*,v^*,\gamma) = \left\{
	\begin{aligned}
		0,~&~\text{if } u^* \leq 1, v^* \leq 0, \gamma >0\\
		+\infty, ~&~ \text{otherwise.}
	\end{aligned}
	\right.
\end{equation*}
On the other hand, for $(x^*, y^*,\alpha) \in \R^3$ and  $u^* \in \R$ such that $(u^*, v^*, \gamma)\in \dom g^c$, we have
\begin{align*}
f^c(u^*-x^*,-y^*,\alpha) &= \sup_\R \left\lbrace c(x,(u^*-x^*, -y^*, \alpha))-f(x)\right\rbrace\\
&=\left\{
	\begin{aligned}
		\sup_{x\geq 0} \left\lbrace x(u^*-x^*-1)\right\rbrace,~&~\text{if } y^* \geq 0,	\alpha>0\\
		+\infty, ~&~ \text{otherwise,}
	\end{aligned}
	\right. \\
&=\left\{
	\begin{aligned}
		0,~&~\text{if } u^*-x^* \leq 1, y^* \geq 0,	\alpha>0\\
		+\infty, ~&~ \text{otherwise.}
	\end{aligned}
	\right.	
\end{align*}
\begin{align*}
(\lambda h)^c(x^*,y^*,\alpha) &= \sup_\R \left\lbrace c(x,(x^*, y^*, \alpha))-\lambda h(x)\right\rbrace\\
&=\left\{
	\begin{aligned}
		\sup_{x\geq 0} \left\lbrace x(x^*+\lambda)\right\rbrace,~&~\text{if } y^* \leq 0,	\alpha>0\\
		+\infty, ~&~ \text{otherwise,}
	\end{aligned}
	\right. \\
&=\left\{
	\begin{aligned}
		0,~&~\text{if } x^* \leq -\lambda, y^* \leq 0,	\alpha>0\\
		+\infty, ~&~ \text{otherwise.}
	\end{aligned}
	\right.	
\end{align*}
We conclude that, for all $(u^*, v^*, \gamma)\in \dom g^c$, the function
$$g^c(u^*,v^*,\gamma)-f^c(u^*-x^*,-y^*,\alpha)-(\lambda h)^c(x^*,y^*,\alpha)$$ is finite (not $-\infty$) if and only if $y^*=0, \alpha >0, x^* \leq -\lambda$ and $x^* \geq u^*-1$, in such a case, its value is zero. Then, if $(x^*, y^*,\alpha) \in \R^3$ and $\lambda \geq 0$
\begin{align*}
\inf_{\dom {g^c}} \left\{ g^c(u^*,v^*,\gamma)-f^c(u^*-x^*,-y^*,\alpha)\right\rbrace - (\lambda h)^c(x^*,y^*,\alpha)\\
=\left\{
	\begin{aligned}
		0,~&~\text{if } y^* = 0, \alpha>0, x^* \leq -\lambda,x^* \geq u^*-1,\\
		-\infty, ~&~ \text{otherwise,}
	\end{aligned}
	\right. \\
\end{align*}
and 
$$v\eqref{eq:Dual_problem_2}=\sup_{\substack{\R_+,Z}}\biggl\{ \inf_{\dom {g^c}} \left\{ g^c(u^*,v^*,\gamma)-f^c(u^*-x^*,-y^*,\alpha)\right\} - (\lambda h)^c(x^*,y^*,\alpha)\biggr\}=0.$$
\begin{flushright}
	$\square$
\end{flushright}
\end{example}

From this example it is clear that conditions ensuring both weak and strong dualities for the dual pair \eqref{eq:FL_Dnew}-\eqref{eq:Primal_problem} are needed, as it used to happen for Fenchel and Lagrange dualities in Section \ref{sec:Cond_WD} and \cite{FV2023}. Let us define the set
\begin{equation}
\label{eq:Set_Kpp_FL}
	\begin{aligned}
	K^{\prime\prime}:= \bigcup_{\RTD_+}~\bigcup_{\dom \lambda h} ~\bigcap_{\dom g^c} \biggl\{ & \epi\bigl(f-c(\cdot,(y^*,v^*,\alpha))\bigr)^c\\
	& -\Bigl(x^*,0,0,g^c(x^*,u^*,\gamma)-(\lambda h)^c(-y^*,-v^*,\alpha)\Bigr)\biggr\},
	\end{aligned}
\end{equation}
which will shortly allow us to derive analogous duality results not only for weak but also for strong duality between \eqref{eq:FL_Dnew}$-$\eqref{eq:Primal_problem}. As expected, this set turns out to be a combination of \cite[Eq.~(2)]{FV2023} and \eqref{eq:Set_Kp}. In addition, it is straightforward to derive the corresponding Fenchel-Lagrange version of Lemma \ref{lem:Lemma_aux} from this manuscript or \cite[Lem.~4.1~ii)]{FV2023}. To avoid unnecessary repetitions, we conclude this section presenting next result where we characterize zero duality gap together with weak and strong dualities for \eqref{eq:Primal_problem}$-$\eqref{eq:FL_Dnew}. We omit the proof since its arguments resemble the ones used in Section \ref{sec:Cond_WD} from this manuscript and in \cite[Sect.~4]{FV2023}.

\begin{theorem}
Let $A$, $K^{\prime\prime}$ and $B$ the sets defined in \eqref{eq:Set_A}, \eqref{eq:Set_Kpp_FL} and \eqref{eq:Set_B}, respectively. The following statements hold
\begin{itemize}
	\item[i)] Weak duality holds for \eqref{eq:Primal_problem}$-$\eqref{eq:FL_Dnew} if and only if 
	\begin{equation*}
		K^{\prime\prime}\cap B\subseteq \epi(f-g+\delta_A)^c\cap B.
	\end{equation*}
	\item[ii)] There exists zero duality gap for \eqref{eq:Primal_problem}$-$\eqref{eq:FL_Dnew} if and only if 
	\begin{equation*}
		\epco(K^{\prime\prime})\cap B = \epi(f-g+\delta_A)^c\cap B.
	\end{equation*}
	\item[iii)] Strong duality holds for \eqref{eq:Primal_problem}$-$\eqref{eq:FL_Dnew} if and only if 
	\begin{equation*}
		K^{\prime\prime}\cap B = \epi(f-g+\delta_A)^c\cap B.
	\end{equation*}
\end{itemize}
\end{theorem}

\section{Conclusions}
\label{sec:Conclusions}
In this paper we present two dual problems for a general DC primal problem with an arbitrary number of inequality convex constraints. Under the even convexity of just one of the functions in the objective, both dual problems turn out to be equivalent. While for one of them weak duality trivially holds, for the other we need extra conditions to ensure it when one of the involved functions is just convex but not necessarily e-convex. We develop characterizations not only for weak duality, but also for zero duality gap and strong duality for both dual pairs and we study the relation between both strong dualities. Finally, we adapt our results from this work and from \cite{FV2023} to, after a derivation of an alternative Fenchel-Lagrange dual problem to the one stated in \cite{FV2017}, study not only weak but also zero duality gap and strong duality between this new Fenchel-Lagrange dual problem and the primal DC problem \eqref{eq:Primal_problem}.

\section*{Disclosure statement}
The authors declare that they have no conflict of interest.

\section*{Funding}
Research partially supported by MICIIN of Spain, Grant AICO/2021/165. 

\bibliographystyle{plain}
\bibliography{biblio.bib}

\end{document}